\documentclass[12pt]{article}
\usepackage[utf8]{inputenc}
\usepackage{amsthm}
\usepackage{amsfonts}
\usepackage{amsmath}
\usepackage{amssymb}
\usepackage{colonequals}
\usepackage{epsfig}
\usepackage{url}
\newtheorem{theorem}{Theorem}
\newtheorem{corollary}[theorem]{Corollary}

\newtheorem{lemma}[theorem]{Lemma}
\newtheorem{observation}[theorem]{Observation}

\renewcommand{\AA}{{\cal A}}

\newcommand{\II}{{\cal I}}
\DeclareMathOperator{\diam}{\mathrm{diam}}
\DeclareMathOperator{\rad}{\mathrm{rad}}
\DeclareMathOperator{\vol}{\mathrm{vol}}
\DeclareMathOperator{\scol}{\mathrm{scol}}
\DeclareMathOperator{\wcol}{\mathrm{wcol}}
\DeclareMathOperator{\col}{\mathrm{col}}
\DeclareMathOperator{\sreach}{\mathrm{sreach}}
\DeclareMathOperator{\wreach}{\mathrm{wreach}}
\DeclareMathOperator{\decr}{\mathrm{decr}}

\newcommand{\footremember}[2]{%
   \footnote{#2}
    \newcounter{#1}
    \setcounter{#1}{\value{footnote}}%
}
\newcommand{\footrecall}[1]{%
    \footnotemark[\value{#1}]%
} 

\title{Weak Coloring Numbers of Intersection Graphs}

\author{
 Zden\v{e}k Dvo\v{r}\'ak\footremember{erccz}{Supported by the ERC-CZ project LL2005 (Algorithms and complexity within and beyond bounded expansion) of the Ministry of Education of Czech Republic.} \\
 {\small Charles University Prague} \\
 {\small \tt rakdver@iuuk.mff.cuni.cz} \\
 \and
 Jakub Pek\'arek\footrecall{erccz} \\
 {\small Charles University Prague} \\
 {\small \tt pekarej@iuuk.mff.cuni.cz} \\
 \and
 Torsten Ueckerdt \\
 {\small Karlsruhe Institute of Technology} \\
 {\small \tt torsten.ueckerdt@kit.edu} \\
 \and 
 Yelena Yuditsky \\
 {\small Universit\'e Libre de Bruxelles} \\
 {\small \tt yuditskyL@gmail.com}\\
}

\begin{document}

\maketitle

\begin{abstract}
 Weak and strong coloring numbers are generalizations of the degeneracy of a graph, where for each natural number $k$, we seek a vertex ordering such every vertex can (weakly respectively strongly) reach in $k$ steps only few vertices with lower index in the ordering.
 Both notions capture the sparsity of a graph or a graph class, and have interesting applications in the structural and algorithmic graph theory.
 Recently, the first author together with McCarty and Norin observed a natural volume-based upper bound for the strong coloring numbers
 of intersection graphs of well-behaved objects in $\mathbb{R}^d$, such as homothets of a centrally symmetric compact convex object, or comparable axis-aligned boxes.
 
 In this paper, we prove upper and lower bounds for the $k$-th weak coloring numbers of these classes of intersection graphs.
 As a consequence, we describe a natural graph class whose strong coloring numbers are polynomial in $k$, but the weak coloring numbers
 are exponential.  We also observe a surprising difference in terms of the dependence of the weak coloring numbers on the dimension
 between touching graphs of balls (single-exponential) and hypercubes (double-exponential).
\end{abstract}

\section{Introduction}

All the graphs we consider are finite, simple and undirected. For the basic graph theoretic notions used in this paper, see \cite{BonMur}. 

Given a linear ordering $\prec$ of vertices of a graph $G$ and an integer $k\ge 0$, a vertex $u$ is \emph{weakly $k$-reachable}
from a vertex $v$ if $u\preceq v$ and there exists a path in $G$ from $v$ to $u$ of length at most $k$ with all internal vertices greater than $u$,
and \emph{strongly $k$-reachable} if there exists such a path with all internal vertices greater than $v$.
Let $\wreach_{G,\prec,k}(v)$ and $\sreach_{G,\prec,k}(v)$ denote the sets of vertices that are weakly and strongly $k$-reachable from $v$,
respectively.  We define \emph{weak and strong coloring numbers} for a given ordering $\prec$ as
\begin{align*}
\wcol_{\prec,k}(G)&=\max_{v\in V(G)} |\wreach_{G,\prec,k}(v)|\\
\scol_{\prec,k}(G)&=\max_{v\in V(G)} |\sreach_{G,\prec,k}(v)|
\end{align*}
The weak and strong coloring numbers of a graph are then obtained by minimizing over all linear orderings of $V(G)$.
\begin{align*}
\wcol_k(G)&=\min_{\prec} \wcol_{\prec,k}(G)\\
\scol_k(G)&=\max_{\prec} \scol_{\prec,k}(G)
\end{align*}

Note that for $k=1$, both $\wreach_{G,\prec,1}(v)\setminus\{v\}$ and $\sreach_{G,\prec,1}(v)\setminus\{v\}$ consist of the neighbors of $v$ that
precede it in the ordering $\prec$, and thus $\scol_1(G)= \wcol_1(G)$ coincide with the \emph{coloring number} $\col(G)$ of the graph~$G$,
equal to the \emph{degeneracy} of $G$ plus one.

One can easily check the following observations:
\begin{align*}
\wcol_1(G)\le \wcol_2(G) &\le \cdots \le \wcol_{\infty}(G);\\
\scol_1(G)\le \scol_2(G) &\le \cdots \le \scol_{\infty}(G)\text{; and}\\
\scol_k(G)&\le \wcol_k(G),
\end{align*}
for any $k\in \mathbb{N}\cup \{\infty\}$.

Let $G$ be a graph. Let $\chi(G)$ and $\Delta(G)$ be the chromatic number and the maximum degree of $G$. One can easily check the following relation between those graph parameters and $\col(G)=\scol_1(G)=\wcol_1(G)$:
\[
\chi(G)\le \col(G)\le \Delta(G)+1.
\]

Weak coloring numbers were introduced by Kierstead and Yang \cite{kierstead2003orderings} and were used in particular to study marking and coloring games on graphs. They also showed the following inequality:
\[
\wcol_k(G)\le (\scol_k(G))^k,
\]
for any $k\in \mathbb{N}\cup \{\infty\}$ and a graph $G$. 

Zhu \cite{zhu2009colouring} showed that the weak and strong coloring numbers can be also used to study important notions of sparsity in a graph and characterize classes of graphs with bounded expansion and nowhere dense classes. A graph class ${\cal C}$ has \emph{bounded expansion} if there is a function $f:\mathbb{Z}^+\rightarrow \mathbb{R}$ such that for every $G\in {\cal C}$ and for any system $B_1,B_2,\ldots, B_t$ of pairwise vertex-disjoint subgraphs of $G$ of radius at most $r$, the minor obtained by contacting each $B_i$ into a vertex
and deleting all other vertices of $G$ has average degree at most $f(r)$. Classes of bounded expansion include planar graphs
and more generally all proper classes closed under taking minors or topological minors. See \cite{nesbook,PPSlecturenotes} for more information on this topic. 

Before proceeding to our results we mention some known bounds on maximum of $\wcol_k(G)$, $G\in {\cal C}$, for some specific classes of graphs ${\cal C}$. For the class ${\cal C}$ of outerplanar graphs the bound is $\Theta(k \log k)$ \cite{JM21}; for the class ${\cal C}$ of planar graphs the upper bound is $O(k^3)$ \cite{vdHOdMQRS16} and  the lower bound is $\Omega(k^2 \log k)$ \cite{JM21}; for the class ${\cal C}$ of graphs of Euler genus $g$ the upper bound is $O_g(k^3)$ \cite{vdHOdMQRS16}; for the class ${\cal C}$ of graphs of tree width at most $r$ the upper bound is $\binom{k+r}{k}$ \cite{covcol}.

\section{Our results} 
Let $S$ be a finite set of objects in $\mathbb{R}^d$.  The \emph{intersection graph} of $S$ is the graph $G$ with $V(G)=S$ and with $uv\in E(G)$
if and only if $u\cap v\neq\emptyset$.  For an integer $t\ge 1$, we say that the set $S$ is \emph{$t$-thin} if every point of $\mathbb{R}^d$
is contained in the interior of at most $t$ objects from $S$; in case $t=1$, we say $S$ is a \emph{touching representation} of $G$.
For example, a famous result of Koebe~\cite{koebe} states that a graph is planar if and only if it has a touching representation
by balls in $\mathbb{R}^2$.  As observed in~\cite{subconvex}, there is a very natural way of bounding the strong coloring numbers for thin intersection graphs of certain classes
of objects by ordering the vertices in a non-increasing order according to the size of the objects that represent the vertices.
In particular, this approach works in case the objects in $S$ are
\begin{itemize}
\item scaled and translated copies of the same centrally symmetric compact convex object (this includes intersection graphs of balls and of axis-aligned hypercubes); or
\item \emph{$b$-ball-like} for some real number $b\ge 1$, i.e., every $v\in S$ is a compact convex set satisfying
$\vol(v)\ge \vol(B(\diam(v)/2))/b$, where $B(a)$ is the ball in $\mathbb{R}^d$ of radius $a$ and $\diam(v)$ is the maximum
distance between any two points of $v$; or
\item \emph{comparable axis-aligned boxes}, i.e, $S$ is a set of axis-aligned boxes with the additional property that
for every $u,v\in S$, a translation of $u$ is a subset of $v$ or vice versa.
\end{itemize}
As we are going to build on this argument, let us give a sketch of it.  A linear ordering $\prec$ of a finite set of compact
objects $S$ is \emph{sizewise} if for all $u,v\in S$ such that $u\prec v$, we have $\diam(u)\ge \diam(v)$.
\begin{lemma}\label{lemma-strong}
Let $d$ and $t$ be positive integers.
Let $S$ be a $t$-thin finite set of compact convex objects in $\mathbb{R}^d$ and let $G$ be the intersection graph of $S$.
Let $\prec$ be a sizewise linear ordering of $S$.  For each integer $k\ge 1$,
\begin{itemize}
\item[(a)] if $S$ consists of scaled and translated copies of the same centrally symmetric object, or if $S$ is a set of comparable axis-aligned boxes,
then $\scol_{\prec,k}(G)\le t(2k+1)^d$, and
\item[(b)] if $S$ consists of $b$-ball-like objects for a real number $b\ge 1$, then $\scol_{\prec,k}(G)\le bt(2k+2)^d$.
\end{itemize}
\end{lemma}
\begin{proof}
Consider a vertex $v\in V(G)$; we need to provide an upper bound on $|\sreach_{G,\prec,k}(v)|$.  For any $m\ge 0$, in case (a) let $B_m(v)$ be the object obtained by scaling $v$ by the factor of $2m+1$, with the center $p$
of $v$ being the fixed point; i.e., $B_m(v)=\{p+(2m+1)(q-p)\colon q\in v\}$.  In case (b),
let $B_m(v)$ be a ball of radius $(m+1)\diam(v)$ centered at an arbitrarily chosen point of $v$.

For each $u\in\sreach_{G,\prec,k}(v)$, observe that $u\cap B_{k-1}(v)\neq\emptyset$, as $u$ is joined to $v$ through a path with at
most $k-1$ internal vertices, each represented by an object smaller or equal to $v$ in size.
In case (a), observe that there exists a translation $u'$ of $v$ such that $u'\subseteq u$ and $u'\cap B_{k-1}(v)\neq\emptyset$.
In case (b), let $u'$ be a scaled translation of $u$ such that $u'\subseteq u$, $u'\cap B_{k-1}(v)\neq\emptyset$, and $\diam(u')=\diam(v)$.
Note that in the former case we have $\vol(u')=\vol(v)=(2k+1)^{-d}\vol(B_k(v))$, and in the latter case we have
\begin{align*}
\vol(u')&=\frac{\diam^d(v)}{\diam^d(u)}\vol(u)\ge \frac{\diam^d(v)}{b\diam^d(u)}\vol(B(\diam(u)/2))\\
&=b^{-1}\vol(B(\diam(v)/2))=b^{-1}(2k+2)^{-d}\vol(B_k(v)).
\end{align*}
In either case, observe that $u'\subseteq B_k(v)$, and since $S$ is $t$-thin, we have
$$\sum_{u\in \sreach_{G,\prec,k}(v)} \vol(u')\le t\vol(B_k(v)).$$
Therefore, $|\sreach_{G,\prec,k}(v)|\le t(2k+1)^d$ in case (a) and $|\sreach_{G,\prec,k}(v)|\le bt(2k+2)^d$ in case (b).
\end{proof}
That is, the strong coloring numbers of these graph classes are polynomial in $k$, with a uniform ordering of vertices that
works for all values of $k$.  For weak coloring numbers, a general upper bound is as follows.
\begin{observation}\label{obs-weak}
For any graph $G$, a linear ordering $\prec$ of its vertices, and an integer $k\ge 1$,
$$\wcol_{\prec,k}(G)\le \sum_{i=1}^k \scol_{\prec,i}(G)\wcol_{\prec,k-i}(G).$$
In particular, if there exists $c>1$ such that $\scol_{\prec,k}(G)\le c^k$ for every $k\ge 1$, then
$\wcol_{\prec,k}(G)\le (2c)^k$ for every $k\ge 1$.
\end{observation}
For graphs from the classes described in Lemma~\ref{lemma-strong}, we obtain an exponential bound on the weak coloring numbers,
more precisely $\wcol_k(G)\le \bigl(2t3^d\bigr)^k$ in case (a) and $\wcol_k(G)\le \bigl(2bt4^d\bigr)^k$ in case (b).

Joret and Wood (see~\cite{espsublin}) conjectured that every class of graphs with polynomial strong coloring numbers also has polynomial
weak coloring numbers.  This turns out not to be the case; Grohe et al.~\cite{covcol} showed that the class of graphs obtained by subdividing
each edge of the graph the number of times equal to its treewidth has superpolynomial coloring numbers.  However, one could still
expect this conjecture to hold for ``natural'' graph classes, and thus we ask whether the weak coloring numbers are polynomial for the
graph classes described in Lemma~\ref{lemma-strong}.  On the positive side, we obtain the following
result.
\begin{theorem}\label{thm-weak}
Let $d$ and $t$ be positive integers.
Let $S$ be a $t$-thin finite set of compact convex objects in $\mathbb{R}^d$ and let $G$ be the intersection graph of $S$.
Let $\prec$ be a sizewise linear ordering of $S$.  For each integer $k\ge 1$:
\begin{itemize}
\item[(a)] If $S$ consists of scaled and translated copies of the same centrally symmetric object,
then $$\wcol_{\prec,k}(G)\le t\max(1,\lceil \log_2 k\rceil)(4k-1)^d\binom{k+t5^d+2}{t5^d+2}.$$
\item[(b)] If $S$ consists of $b$-ball-like objects for a real number $b\ge 1$, then
$$\wcol_{\prec,k}(G)\le tb\max(1,\lceil \log_2 k\rceil)(4k)^d\binom{k+tb6^d+2}{tb6^d+2}.$$
\end{itemize}
Moreover, there exists $k_0$ (depending only on $d$) such that if $S$ consists of balls, then for every $k\ge k_0$,
$$\wcol_{\prec,k}(G)\le t\max(1,\lceil \log_2 k\rceil)(4k-1)^d\binom{k+2t+2}{2t+2}.$$ 
\end{theorem}

This theorem is qualitatively tight in several surprising aspects, summarized in the following result.
\begin{theorem}\label{thm-lb}
For every positive integer $k$:
\begin{itemize}
\item[(i)] There exists a touching graph $F_k$ of comparable axis-aligned boxes in $\mathbb{R}^3$ such that
$\wcol_{2k}(F_k)\ge 2^{k+1}-1$.
\item[(ii)] For every $t$, there exists a $t$-thin set of axis-aligned squares in $\mathbb{R}^2$ whose intersection graph $H_{k,t}$
satisfies $\wcol_{2k}(H_{k,t})\ge\binom{k+t}{t}$.
\item[(iii)] For every $d\ge 1$, the graph $H_{k,2^d-1}$ can also be represented as a touching graph of axis-aligned hypercubes
in $\mathbb{R}^{d+2}$.
\end{itemize}
\end{theorem}
That is:
\begin{itemize}
\item[(i)] The (rather natural) class of touching graphs of comparable axis-aligned boxes in $\mathbb{R}^3$
has polynomial strong coloring numbers but exponential weak coloring ones.  Let us remark that touching graphs
of rectangles in $\mathbb{R}^2$ are obtained from planar graphs by adding crossing edges into faces of size four
(when four of the boxes share corners), and such graphs have polynomial weak coloring numbers (this follows e.g.
from their product structure~\cite{dujmovic2019graph}).
\item[(ii)] Unlike in Lemma~\ref{lemma-strong}, in dimension at least two the dependence of the weak coloring numbers in Theorem~\ref{thm-weak}
on the thinness $t$ must indeed be in the exponent, and not just in the multiplicative constant.
Let us also remark that $t$-thin intersection graphs of intervals in $\mathbb{R}$ are interval graphs of clique number at most $2t$.
As was pointed to us by Gwena\"el Joret, any interval graph of clique number $\omega$ satisfies
$\wcol_k(G)\le \binom{\omega+1}{2}(k+1)$, as shown by an ordering obtained by placing first the vertices of a maximal system of pairwise disjoint cliques
of size $\omega$ and then recursively processing the remainder of the graph which has clique number smaller than $\omega$.
\item[(iii)] In the case (a) of Theorem~\ref{thm-weak}, and in particular for the touching graphs of axis-aligned hypercubes, the exponent must
be exponential in the dimension, in a contrast to the case of touching graphs of balls.
\end{itemize}

\section{Upper bounds}

In order to prove Theorem~\ref{thm-weak} for all the classes at once, let us formulate an
abstract graph property $P(f,a,e)$ on which the proof is based.
For a graph $G$, a function $r\colon V(G)\to \mathbb{R}^+$ and $u,v\in V(G)$, let us define
$\lambda_r(u,v)$ as the minimum of $\sum_{x\in V(Q)\setminus \{u,v\}} r(x)$ over all paths $Q$ from $u$ to $v$ in $G$.
For a function $f\colon \mathbb{Z}_0^+\to\mathbb{Z}^+$ and positive integers $a$ and $e$, we say that $(G,r)$ has \emph{the property $P(f,a,e)$} if
\begin{itemize}
\item[(i)] for each $v\in V(G)$ and integers $s\ge 1$ and $p\ge 0$,
there are at most $f(p)$ vertices $u\in V(G)$ such that $r(u)\ge sr(v)$ and $\lambda_r(u,v)\le psr(v)$, and
\item[(ii)] for each $v\in V(G)$ and each positive integer $s$, every sequence $u_1$, $u_2$, \ldots\ of distinct vertices of $G$ such that
$\lambda_r(u_i,v)\le sr(v)$ and $r(u_i)\ge a^isr(v)$ for each $i$ has length at most $e$.
\end{itemize}
Let us remark that $P(f,a,e)$ implies $P(f,a',e)$ for every $a'\ge a$, and (i) implies (ii) with $a=1$ and $e=f(1)$.
The following lemma is proved similarly to Lemma~\ref{lemma-strong}.

\begin{lemma}\label{lemma-prop}
Let $d$ and $t$ be positive integers.
Let $S$ be a $t$-thin finite set of compact convex objects in $\mathbb{R}^d$ and let $G$ be the intersection graph of $S$.
For $v\in V(G)$, let $r(v)=\diam(v)$.
\begin{itemize}
\item[(a)] If $S$ consists of scaled and translated copies of the same centrally symmetric object,
then $(G,r)$ has the property $P(p\mapsto t(2p+3)^d,1,t5^d)$.
\item[(b)] If $S$ consists of $b$-ball-like objects for $b\ge 1$, then $(G,r)$ has the property $P(p\mapsto tb(2p+4)^d,1,tb6^d)$.
\item[(c)] If $S$ consists of balls, then there exists $a$ such that $(G,r)$ has the property $P(p\mapsto t(2p+3)^d,a,2t)$.
\end{itemize}
\end{lemma}
\begin{proof}
Consider a vertex $v\in V(G)$ and integers $s\ge 1$ and $p\ge 0$.  For any $m\ge 0$, in cases (a) and (c) let $B_m(v)$ be
the object obtained by scaling $v$ by the factor of $2m+1$, with the center of $v$ being the fixed point.
In case (b), let $B_m(v)$ be a ball of radius $(m+1)\diam(v)$ centered at an arbitrarily chosen point of $v$.
Let $U$ be the set of vertices $u\in V(G)$ such that $r(u)\ge sr(v)$ and $\lambda_r(u,v)\le psr(v)$.
Observe that for any $u\in U$, we have $u\cap B_{ps}(v)\neq\emptyset$.
Let $u'$ be a scaled translation of $u$ such that $u'\subseteq u$, $u'\cap B_{ps}(v)\neq\emptyset$, and $\diam(u')=s\diam(v)$.
For each $m\ge 0$, in cases (a) and (c), we have
$$\vol(u')=s^d\vol(v)=\bigl(\tfrac{s}{2m+1}\bigr)^d\vol(B_m(v)),$$
and in case (b) we have
$$\vol(u')\ge b^{-1}s^d\vol(B(\diam(v)/2))=b^{-1}\bigl(\tfrac{s}{2m+2}\bigr)^d\vol(B_m(v)).$$
In either case, we have $u'\subseteq B_{(p+1)s}(v)$, and since $S$ is $t$-thin, it follows that
$$|U|\le t\bigl(\tfrac{2(p+1)s+1}{s}\bigr)^d\le t(2p+3)^d$$
in cases (a) and (c), and
$$|U|\le tb\bigl(\tfrac{2(p+1)s+2}{s}\bigr)^d\le tb(2p+4)^d$$
in case (b).  Hence, the part (i) of the property $P(f,a,e)$ is verified,
and by the observations made before the lemma, this finishes the proof for the cases (a) and (b).

Let us now consider the part (ii) in case (c).  Let $Q$ be a half-space whose boundary hyperplane touches $B_s(v)$
and is otherwise disjoint from $B_s(v)$.
There exists $l$ such that $\vol(Q\cap B_{ls}(v))\ge \bigl(\tfrac{1}{2}-\tfrac{1}{6t}\bigr)\vol(B_{ls}(v))$; let us fix smallest such $l$.
For $a\ge 1$, let
$C_a$ be a ball touching $B_s(v)$ of radius $as\rad(v)$. I.e. $C_a \subseteq Q$.  Note that
$$\lim_{a\to\infty} \frac{\vol(C_a\cap B_{ls}(v))}{\vol(B_{ls}(v))}=\frac{\vol(Q\cap B_{ls}(v))}{\vol(B_{ls}(v))},$$
and thus there exists $a$ such that $\vol(C_a\cap B_{ls}(v))\ge \bigl(\tfrac{1}{2}-\tfrac{1}{7t}\bigr)\vol(B_{ls}(v))$;
let us fix smallest such $a$.

Consider a sequence $u_1$, $u_2$, \ldots, $u_n$ of distinct vertices of $G$ such that $\lambda_r(u_i,v)\le sr(v)$ and $r(u_i)\ge a^isr(v)$ for each $i$.
In particular, note that $\rad(u_i)\ge \rad(C_a)$ for each $i$.  From the observation made in the first paragraph of the proof,
we have $u_i\cap B_s(v)\neq\emptyset$, and it follows that
$$\frac{\vol(u_i\cap B_{ls}(v))}{\vol(B_{ls}(v))}\ge \frac{\vol(C_a\cap B_{ls}(v))}{\vol(B_{ls}(v))}\ge \tfrac{1}{2}-\tfrac{1}{7t}.$$
Since $S$ is $t$-thin and $n$ is an integer, this implies $n\le 2t$, verifying the part (ii) of the property
$P(p\mapsto t(2p+3)^d,a,2t)$.
\end{proof}

To bound the weak coloring numbers, we need the following result about graphs of bounded pathwidth
which appears in a stronger form (for treewidth) in van den Heuvel et al.~\cite{van2017generalised}.  For us, it is convenient to
state the result as follows (without explicitly defining pathwidth), and thus we include the proof for completeness.
A path $P=v_1v_2\ldots v_m$ in a graph $G$ with a linear ordering $\prec$ of vertices is \emph{decreasing}
if $v_1\succ v_2\succ\cdots\succ v_m$.  For each $v\in V(G)$, we define $\decr_{G,\prec,k}(v)$ as the
set of vertices reachable from $v$ by decreasing paths of length at most $k$.

\begin{lemma}\label{lemma-pw}
Let $k$ and $w$ be non-negative integers.  Let $\prec$ be a linear ordering of the vertices of a graph $G$.
If for every $x\in V(G)$, at most $w$ vertices $y\prec x$ have a neighbor $y'\succeq x$, then
$|\decr_{G,\prec,k}(v)|\le \binom{k+w}{w}$ for every $v\in V(G)$.
\end{lemma}
\begin{proof}
Without loss of generality, we assume that if $yy'\in E(G)$ and $y\prec y'$, then $y$ is also adjacent to
all vertices $x$ such that $y\prec x\prec y'$.  Indeed, adding such an edge $yx$ does not violate the assumptions
and can only increase $|\decr_{G,\prec,k}(v)|$.

The proof is by induction on $k+w$.  Note that $|\decr_{G,\prec,0}(v)|=1$, and thus we can assume $k\ge 1$.
If no neighbor of $v$ is smaller than $v$, then $|\decr_{G,\prec,k}(v)|=1$,
and thus the claim of the lemma holds.  Hence, we can assume $v$ has such a neighbor, and in particular $w\ge 1$.
Let $z$ be the smallest neighbor of $v$.  Let $G'$ be the subgraph of $G$ induced by the vertices greater than $z$
and smaller or equal to $v$.  Since $z$ is adjacent to all the vertices of $G'$, note that for each $x\in V(G')$,
at most $w-1$ vertices $y\prec x$ of $G'$ have a neighbor $y'\succeq x$ in $G'$.

Consider now a vertex $u\in\decr_{G,\prec,k}(v)$, and let $Q$ be a decreasing path of length at most $k$ from $v$ to $u$.
If $z\prec u$, then $Q$ is also a decreasing path in $G'$, and thus $u\in \decr_{G',\prec,k}(v)$.  Note that
$|\decr_{G',\prec,k}(v)|\le \binom{k+w-1}{w-1}$ by the induction hypothesis.
If $u\prec z$, consider the edge $u'z'$ of $Q$ such that $u'\prec z$ and $z\preceq z'$. Note that $u'$ is not adjacent
to $v$ by the minimality of $z$, and thus $z'\neq v$.  Moreover, by the assumption made in the first paragraph,
$u'z\in E(G)$.  Hence, $u$ is reachable from $v$ by the decreasing path of length at most $k$ starting with $vzu'$
and continuing along $Q$, and thus $u\in \decr_{G,\prec,k-1}(z)$.
If $u=z$, then we also have $u\in \decr_{G,\prec,k-1}(z)$.
By the induction hypothesis, we have $|\decr_{G,\prec,k-1}(z)|\le \binom{k+w-1}{w}$.

Therefore,
\begin{align*}
|\decr_{G,\prec,k}(v)|&=|\decr_{G',\prec,k}(v)|+|\decr_{G,\prec,k-1}(z)|\\
&\le\binom{k+w-1}{w-1}+\binom{k+w-1}{w}=\binom{k+w}{w}.
\end{align*}
\end{proof}

We use the following corollary, obtained by applying Lemma~\ref{lemma-pw} to the graph obtained by contracting each interval to a single vertex.

\begin{corollary}\label{cor-copw}
Let $w$, $k$, and $m$ be non-negative integers.
Let $\prec$ be a linear ordering of vertices of a graph $H$, and let $\II=\{I_i\colon i=0,1,\ldots\}$ be a partition of $V(H)$ into consecutive
intervals in this ordering, where for every $i<j$, $u\in L_i$, and $v\in L_j$, we have $u\succ v$ (note the reverse ordering of the indices).
Suppose that for each $i\ge 0$, we have $|L_i|\le m$ and there are at most $w$ indices $j>i$ such that
a vertex of $L_j$ has a neighbor in $L_0\cup L_1\cup \cdots\cup L_i$.  Then $|\decr_{H,\prec,k}(v)|\le m\binom{k+w}{w}$ for each $v\in V(H)$.
\end{corollary}

Theorem~\ref{thm-weak} now follows from Lemma~\ref{lemma-prop} and the following theorem.
\begin{theorem}
Let $f\colon \mathbb{Z}_0^+\to\mathbb{Z}^+$ be a function and let $a$ and $e$ be positive integers. 
For a graph $G$ and a function $r\colon V(G)\to\mathbb{R}^+$,
let $\prec$ be a linear ordering of $V(G)$ such that if $u\prec v$, then $r(u)\ge r(v)$.
If $(G,r)$ has the property $P(f,a,e)$, then
$$\wcol_{\prec,k}(G)\le \max(1,\lceil \log_2 k\rceil)f(2k-2)\binom{k+e+2}{e+2}$$
for every integer $k\ge a$.
\end{theorem}
\begin{proof}
Consider any integer $k\ge a$ and a vertex $v\in V(G)$; we are going to bound the number of vertices weakly $k$-reachable from $v$.
Note that for $k=1$, $\wreach_{G,\prec,1}(v)$ consists of the vertices $x\in V(G)$ such that $r(x)\ge r(v)$ and $\lambda_r(v,x)=0$,
and thus $|\wreach_{G,\prec,1}|\le f(0)$ by the part (i) of the property $P(f,a,e)$ with $s=1$ and $p=0$.
Hence, we can assume that $k\ge 2$.

Let $H$ be the graph with the vertex set $\wreach_{G,\prec,k}(v)$, such that for $x,y\in V(H)$ with $x\prec y$, we have
$xy\in E(H)$ if and only if there exists a path $Q$ of length at most $k$ in $G$ from $v$ to $x$ such that $y\in V(Q)$
and all the internal vertices of the subpath of $Q$ between $x$ and $y$ are greater than $y$.  Let $\ell(xy)$
denote the minimum length of the subpath between $x$ and $y$ over all paths $Q$ satisfying these conditions.
Observe that for every edge $e'$ of $H$, there exists a decreasing path $D$ from $v$ containing the edge $e'$
such that $\sum_{e\in E(D)} \ell(e)\le k$.  Moreover, $V(H)=\decr_{H,\prec,k}(v)$.

For $i\ge 0$, let $L_i$ consist of the vertices $x\in V(H)$ such that $k^ir(v)\le r(x)<k^{i+1}r(v)$;
in particular, $v\in L_0$.
Let $c=\lceil \log_2 k\rceil$ and further partition $L_i$ into $L_{i,1},\ldots,L_{i,c}$, where $L_{i,b}$ consists of the vertices $x \in L_i$ with $2^{b-1}k^ir(v) \leq r(x) < 2^bk^ir(v)$ for $b=1,\ldots,c$.
Consider any vertex $x\in L_{i,b}$.  Since $x$ is weakly $k$-reachable from $v$ and
$r(x)<2^bk^ir(v)$, we have $\lambda_r(v,x)<(k-1)2^bk^ir(v)$.  Moreover, $r(x)\ge 2^{b-1}k^ir(v)$,
and thus by the part (i) of the property $P(f,a,e)$ with $s=2^{b-1}k^i$ and $p=2(k-1)$,
we conclude $|L_{i,b}|\le f(2k-2)$ for each $b \in \{1,\ldots,c\}$.
Hence, we have $|L_i| = |L_{i,1}|+\cdots+|L_{i,c}| \leq cf(2k-2) = \lceil \log_2 k\rceil f(2k-2)$.

Let $j_{-1}<j_0<j_1< \cdots <j_{w-2}$ be all indices such that $j_{-1}>i$ and for each $m\in\{-1,\ldots, w-2\}$, a vertex $u_m\in L_{j_m}$ has 
a neighbor $y_m\in L_0\cup\cdots\cup L_i$ for each $m$.  For $m=1,\ldots,w-2$,
since there exists a decreasing path $D$ from $v$ containing the edge $u_my_m$ such that
$\sum_{e\in E(D)} \ell(e)\le k$, there exists a path $Q$ in $G$ from $v$ to $u_m$ of length at most $k$ such
that $r(x)\le r(y_m)<k^{i+1}r(v)$ for every internal vertex $x$ of $Q$.  Consequently,
we have $\lambda_r(v,u_m)\le (k-1)k^{i+1}r(v)\le sr(v)$ for $s=k^{i+2}$.
Moreover, note that $j_m\ge i+2+m$, and thus $r(u_m)\ge k^{i+2+m}r(v)\ge a^msr(v)$.  By part (ii) of the
property $P(f,a,e)$, we conclude that $w\le e+2$.

Hence, Corollary~\ref{cor-copw} implies that
$$|\wreach_{G,\prec,k}(v)|=|\decr_{H,\prec,k}(v)|\le \lceil \log_2 k\rceil f(2k-2)\binom{k+e+2}{e+2}$$
for each $v\in V(G)$.
\end{proof}

\section{Lower bounds}

\begin{figure}
\begin{center}\includegraphics[width=0.6\textwidth]{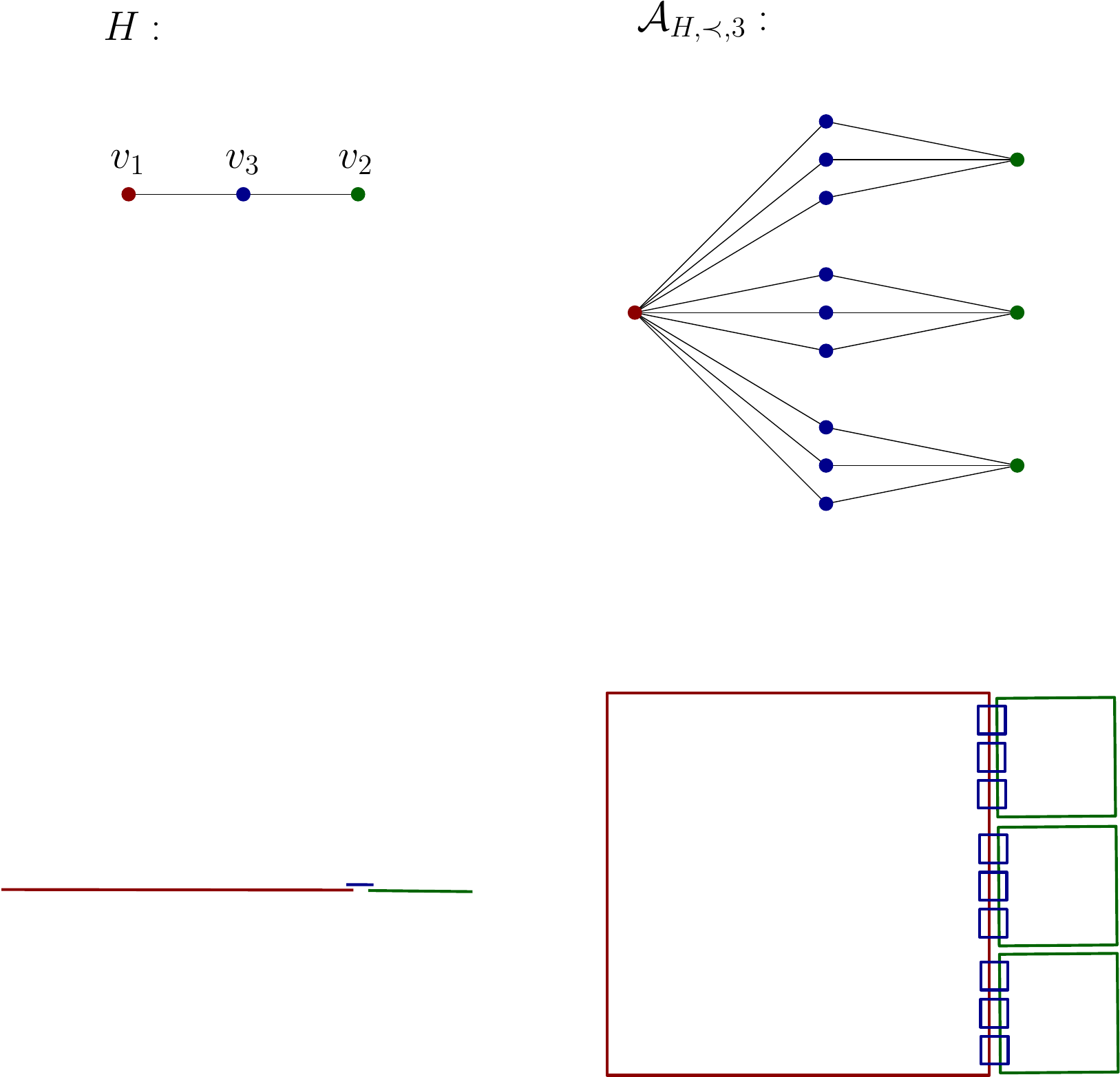}\end{center}
\caption{The graph $\AA_{H,\prec,3}$ and its intersection representation by squares.}\label{fig-constr}
\end{figure}

It is relatively easy to construct intersection graphs with large weak coloring numbers with respect to a fixed coloring.
The following construction (illustrated in the top part of Figure~\ref{fig-constr}) enables us to turn such graphs into graphs that have large weak coloring numbers
with respect to every ordering.  Let $H$ be a graph and $\prec$ a linear ordering of its vertices.
Let $v_1\prec \cdots\prec v_n$ be the vertices of $H$.  Let $m$ be a positive integer and let
$T$ be the complete rooted $m$-ary tree of depth $n-1$.  For $i\in\{1,\ldots, n\}$,
let $T(v_i)$ be the set of vertices of $T$ at distance exactly $i-1$ from the root.
The graph $\AA_{H,\prec,m}$ has vertex set $V(T)$, with vertices $x\in T(v_i)$ and $y\in T(v_j)$
adjacent if and only if $i\neq j$, $v_iv_j\in E(H)$, and $x$ is an ancestor of $y$ in $T$ or vice versa.
We say that $T$ is the \emph{scaffolding} of $\AA_{H,\prec,m}$.

\begin{lemma}\label{lemma-ordtoall}
Let $k$ and $m$ be positive integers.  Let $H$ be a graph and $\prec$ a linear ordering of its vertices.
Suppose that for each $v\in V(H)$, the graph $H[\{u\in V(H) \colon v\preceq u\}]$ is connected and has diameter
at most $k$.  Then
$$\wcol_k(\AA_{H,\prec,m})\ge\min\bigl(m,\wcol_{\prec,k}(H)\bigr).$$
\end{lemma}
\begin{proof}
Consider any linear ordering $\vartriangleleft$ of the vertices of $\AA_{H,\prec,m}$.
Let $T$ be the scaffolding of $\AA_{H,\prec,m}$ and suppose first that there exists
a non-leaf vertex $z\in V(T)$ such that all children $z_1$, \ldots, $z_m$ of $z$ in $T$
are smaller than $z$ in the ordering $\vartriangleleft$.  For $i=1,\ldots, m$, let
$A_i$ be the subgraph of $\AA_{H,\prec,m}$ induced by $z$, $z_i$, and all descendants
of $z_i$ in $T$.  Let $v$ be the vertex of $H$ such that $z\in T(v)$; since the graph
$H[\{u\in V(H)\colon v\preceq u\}]$ has diameter at most $k$, every vertex of $A_i$ is at distance
at most $k$ from $z$.  Since $z_i\vartriangleleft z$, we conclude that a vertex of $A_i$
distinct from $z$ is weakly $k$-reachable from $z$.  Since this is the case for each $i\in\{1,\ldots, m\}$
and the subgraphs $A_1$, \ldots, $A_m$ intersect only in $z$, it follows that
$$\wcol_{\vartriangleleft,k}(\AA_{H,\prec,m})\ge |\wreach_{\AA_{H,\prec,m},\vartriangleleft,k}(z)|\ge m.$$

Hence, we can assume that each non-leaf vertex $z$ of $T$ has a child which is greater than $z$ in the
ordering $\vartriangleleft$.  Consequently, $T$ contains a path $u_1u_2\ldots u_n$ from the root to a leaf
such that $u_1\vartriangleleft \cdots \vartriangleleft u_n$.  The subgraph $A$ of $\AA_{H,\prec,m}$ induced by
$\{u_1,\ldots, u_n\}$ with ordering $\vartriangleleft$ is isomorphic to $H$ with ordering $\prec$, and thus
$$\wcol_{\vartriangleleft,k}(\AA_{H,\prec,m})\ge \wcol_{\vartriangleleft,k}(A)=\wcol_{\prec,k}(H).$$
\end{proof}

Moreover, assuming $H$ has a sufficiently generic representation by comparable axis-aligned boxes, we can also find
such a representation for $\AA_{H,\prec,m}$.  Given an axis-aligned box $v$ in $\mathbb{R}^d$ and $i\in\{1,\ldots,d\}$,
let $\ell_i(v)$ denote the length of $v$ in the $i$-th coordinate.  We say that a sequence $v_1$, \ldots, $v_n$ of axis-aligned boxes is
\emph{$m$-shrinking} if $\ell_d(v_i) > m\ell_d(v_{i+1})$ holds for $1\le i\le n-1$.  See the bottom part of Figure~\ref{fig-constr}
for an illustration of the following construction.

\begin{lemma}\label{lemma-reproa}
Let $d$, $t$ and $m$ be positive integers.
Let $S$ be a $t$-thin finite set of comparable axis-aligned boxes in $\mathbb{R}^d$ and let $H$ be the intersection graph of $S$.
Let $T$ be the scaffolding of $\AA_{H,\prec,m}$.  Let $\prec$ be a sizewise linear ordering of $S$ and
let $v_1$, \ldots, $v_n$ be the sequence of vertices of $H$ in this order.
If this sequence is $m$-shrinking, then $\AA_{H,\prec,m}$ is the intersection graph of a $t$-thin set of
comparable axis-aligned boxes in $\mathbb{R}^{d+1}$, where for $v\in V(H)$ and $u\in T(v)$,
$u$ is the product of $v$ with an interval of length $\ell_d(v)$.
\end{lemma}
\begin{proof}
Let $\varepsilon>0$ be small enough so that $\ell_d(v_i) \ge m(\ell_d(v_{i+1})+\varepsilon)$ holds for $1\le i\le n-1$. 
For each non-leaf vertex $z$ of $T$, assign labels $0$, \ldots, $m-1$ to the edges from $z$ to the children of $z$ 
in any order; let $l(e)$ denote the label assigned to the edge $e$.  For a vertex $y$ of $T$, if $y_1y_2\ldots y_c$ is
the path in $T$ from the root to $y$, then let $l(y)=(l(y_1y_2),l(y_2y_3),\ldots,l(y_{c-1}y_c))$.  Note that $y$ is an ancestor
of a vertex $x$ in $T$ if and only if $l(y)$ is a prefix of $l(x)$.  Let $s(y)=\sum_{i=1}^{c-1} (l(y))_i(\ell_d(v_{i+1})+\varepsilon)$,
and let $I(y)$ be the interval $[s(y),s(y)+\ell_d(v_c)]$.  Observe that if $y$ is an ancestor of a vertex $x$ in $T$, then
$I(x)\subset I(y)$, and if $x$ is neither an ancestor nor a descendant of $y$ in $T$, then $I(x)\cap I(y)=\emptyset$.

Hence, letting each vertex $y$ at distance $c-1$ from the root of $T$ be represented by the box $v_c\times I(y)$ in $\mathbb{R}^{d+1}$,
we obtain a $t$-thin intersection representation of $\AA_{H,\prec,m}$ as described in the statement of the lemma.
\end{proof}

To verify the assumptions of Lemma~\ref{lemma-ordtoall}, the following concept is useful.
Let $\prec$ be a linear ordering of vertices of a graph $G$.  A \emph{decreasing spanning tree}
is a spanning tree $T$ of $G$ rooted in the maximum vertex such that any path in $T$ starting in the root
is decreasing.
\begin{lemma}\label{lemma-verass}
Let $k\ge 0$ be an integer.
Let $\prec$ be a linear ordering of vertices of a graph $G$.  If $G$ has a decreasing spanning tree $T$ of depth at most $k$,
then $\wcol_{\prec,k}(G)=|V(G)|$, and for each $v\in V(G)$, the graph $G[\{u\in V(H)\colon v\preceq u\}]$ is connected and has diameter at most $2k$.
\end{lemma}
\begin{proof}
Let $z$ be the maximum vertex of $G$.  Since $T$ is decreasing and has depth at most $k$, we
have $\wcol_{\prec,k}(G)\ge |\wreach_{G,\prec,k}(z)|=|V(G)|$.  Moreover, for each $v\in V(G)$, letting
$C_v=\{u\in V(H)\colon v\preceq u\}$, observe that for each $x\in C_v$, all ancestors of $x$ also belong to $C_v$.
Hence, $T[C_v]$ is a spanning tree of $G[C_v]$ of depth at most $k$, and thus $G[C_v]$ is connected and has diameter at most $2k$.
\end{proof}

We now find some basic graphs to which we can apply the construction.
\begin{figure}
\begin{center}\includegraphics[width=0.9\textwidth]{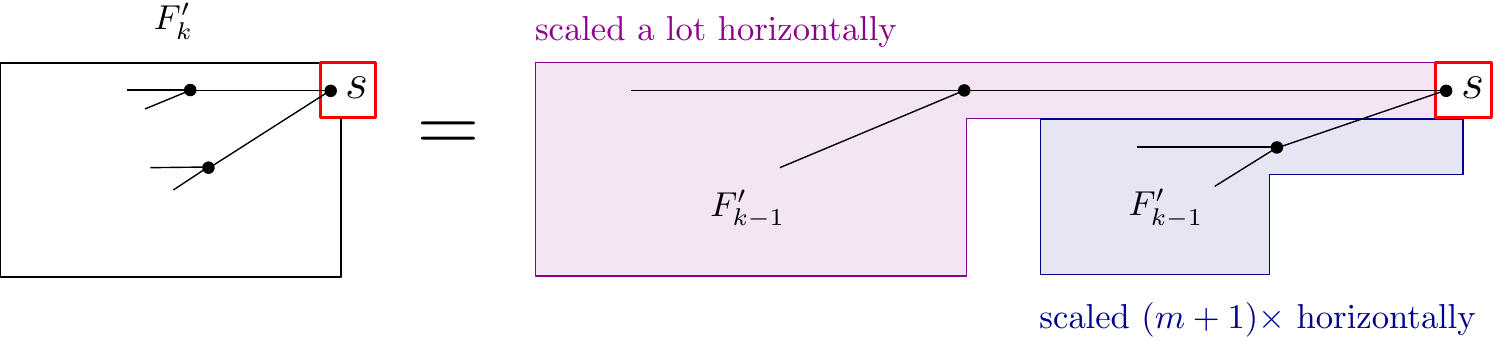}\end{center}

\caption{The construction from Lemma~\ref{lemma-tree}.}\label{fig-tree}
\end{figure}

\begin{lemma}\label{lemma-tree}
For all integers $k\ge 0$ and $m\ge 1$, there exists a graph $F'_k$ with $2^{k+1}-1$ vertices represented as the touching graph
of an $m$-shrinking sequence of comparable axis-aligned rectangles in $\mathbb{R}^2$,
such that $F'_k$ has a spanning tree of depth at most $k$ decreasing in the sizewise ordering.
\end{lemma}
\begin{proof}
We proceed by the induction on $k$.  For each $k$, we construct a representation of $F'_k$ where the last vertex is
represented by a unit square and the rest of the representation is contained in the lower left
quadrant starting from the middle of the upper side of $s$.  The second coordinate (relevant for the definition
of an $m$-shrinking sequence) is the horizontal one.  In the vertical coordinate, all rectangles have length $1$.
See Figure~\ref{fig-tree} for an illustration of the construction.

The graph $F'_0$ is a single vertex represented by $s$.  For $k\ge 1$, to obtain a representation of $F'_k$, we scale the
representation of $F'_{k-1}$ in the horizontal direction by the factor of $m+1$ and place it so that its
upper right corner is the middle of the lower side of $s$.  Then we add another copy of a representation of $F'_{k-1}$,
scaled in the horizontal direction so that all its rectangles are more than $m$ times longer than the already
placed ones and so that when we place its upper right corner at the upper left corner of $s$, their interiors are disjoint
from the already placed rectangles.

Observe that $F'_k$ contains a spanning complete binary tree of depth $k$ rooted in $s$, with the vertices along
each path from the root increasing in size, and thus decreasing in the sizewise ordering. 
\end{proof}

\begin{figure}
\begin{center}\includegraphics[width=0.9\textwidth]{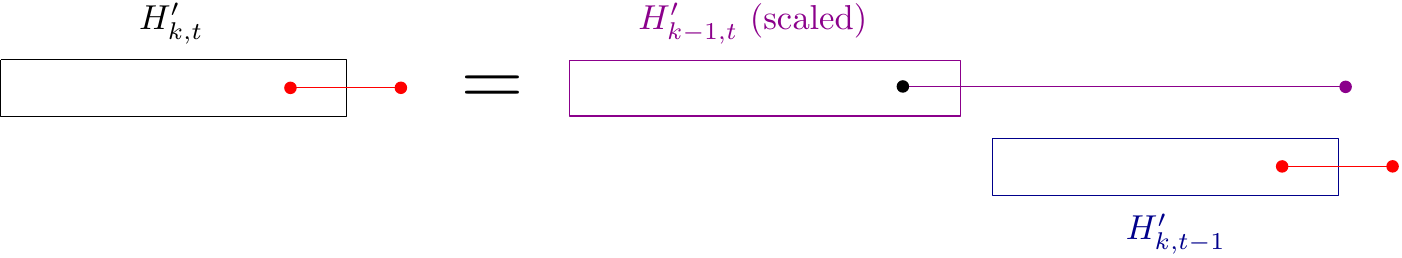}\end{center}

\caption{The construction from Lemma~\ref{lemma-thin}.}\label{fig-thin}
\end{figure}

\begin{lemma}\label{lemma-thin}
For all integers $k\ge 0$ and $m,t\ge 1$, there exists a graph $H'_{k,t}$ with $\binom{k+t}{t}$ vertices
represented by a $t$-thin $m$-shrinking sequence of intervals in $\mathbb{R}$,
such that $H'_{k,t}$ has a spanning tree of depth at most $k$ decreasing in the sizewise ordering.
Furthermore, $H'_{k,t}$ is properly $(t+1)$-colorable.
\end{lemma}
\begin{proof}
We construct a representation of $H'_{k,t}$ with the additional property that the right end of the smallest interval
is the strictly rightmost point of the whole representation.  See Figure~\ref{fig-thin} for an illustration of the construction.

We proceed by the induction on $k+t$.  If $k=0$, the representation of $H'_{k,t}$ consists of a single unit interval.
If $t=1$, then the representation consists of an $m$-shrinking sequence of $k+1$ intervals intersecting only in endpoints.
Hence, suppose that $k\ge 1$ and $t\ge 2$.  Then the representation consists of the representation $A$ of $H'_{k,t-1}$
and of the representation $B$ of $H'_{k-1,t}$ scaled so that all its intervals are more than $m$ times longer than all
intervals in $A$ and so that when we place the rightmost point of $B$ slightly to the left of the rightmost point of $A$,
only the smallest interval of $B$ intersects $A$.

Observe that $H'_{k,t}$ has a spanning tree of depth $k$ rooted in the smallest vertex, with the vertices along
each path from the root increasing in size, and thus decreasing in the sizewise ordering.
Finally, note that $H'_{k,t}$ is an interval graph with clique number at most $t+1$.  Since interval graphs are perfect,
$H'_{k,t}$ is properly $(t+1)$-colorable.
\end{proof}

As a final ingredient, we note that we can trade thinness for dimension.

\begin{lemma}\label{lemma-trade}
For a positive integer $d$, let $S=\{v_1,\ldots,v_n\}$ be a finite set of hypercubes in $\mathbb{R}^d$,
and let $G$ be the intersection graph of $S$.  For any set $Y\subseteq\{1,\ldots,n\}$, there exists a set
$\{u_1,\ldots,u_n\}$ of hypercubes in $\mathbb{R}^{d+1}$ whose intersection graph is isomorphic to $G$
via the isomorphism mapping $u_i$ to $v_i$ for each $i$, such that
\begin{itemize}
\item for $1\le i<j\le n$, if $v_i$ and $v_j$ have disjoint interiors, then $u_i$ and $u_j$ have disjoint interiors, and
\item for $i\in Y$ and $j\in \{1,\ldots,n\}\setminus Y$, the hypercubes $u_i$ and $u_j$ have disjoint interiors.
\end{itemize}
\end{lemma}
\begin{proof}
For $i\in Y$, we set $u_i=v_i\times [0,\ell_1(v_i)]$.  For $i\in \{1,\ldots,n\}\setminus Y$, we set
$u_i=v_i\times [0,-\ell_1(v_i)]$.  Note that the intersection of the representation with the hyperplane defined by the last
coordinate being $0$ is equal to $S$, and thus indeed the intersection graph of $S'$ is isomorphic to $G$
as described.
\end{proof}

\begin{corollary}\label{cor-tradecol}
Let $c\ge 0$ and $d\ge 1$ be integers.  If $G$ is a graph of chromatic number at most $2^c$ representable
as an intersection graph of hypercubes in $\mathbb{R}^d$, then $G$ is also representable as a touching graph
of hypercubes in $\mathbb{R}^{d+c}$.
\end{corollary}
\begin{proof}
Let $V(G)=\{v_1,\ldots,v_n\}$, and let $\varphi\colon V(G)\to \{0,1\}^c$ be a proper coloring of $G$.
By repeatedly applying Lemma~\ref{lemma-trade} for sets $Y_1$, \ldots, $Y_c$, where $Y_b=\{i\in \{1,\ldots,n\} \colon  (\varphi(v_i)_b=0\}$ for $b\in\{1,\ldots,c\}$,
we obtain a representation of $G$ as an intersection graph of hypercubes $u_1$, \ldots, $u_n$ in $\mathbb{R}^{d+c}$
with the property that for $1\le i<j\le n$, if $\varphi(v_i)\neq\varphi(v_j)$, then $u_i$ and $u_j$ have disjoint interiors.
If $\varphi(v_i)=\varphi(v_j)$, then since $\varphi$ is a proper coloring, we have $v_iv_j\not\in E(G)$, and thus
the hypercubes $u_i$ and $u_j$ are disjoint.  Consequently, the hypercubes $u_1$, \ldots, $u_n$ have pairwise disjoint interiors.
\end{proof}

We are now ready to give the lower bounds.
\begin{proof}[Proof of Theorem~\ref{thm-lb}]
We prove each point separately:
\begin{itemize}
\item[(i)] Let $F'_k$ be the graph obtained in Lemma~\ref{lemma-tree}, represented
as a touching graph of an $m$-shrinking sequence of axis-aligned rectangles for $m=2^{k+1}-1$.
Let $\prec$ be the sizewise ordering of $F'_k$.
By Lemma~\ref{lemma-verass}, we have $\wcol_{\prec,k}(F'_k)=|V(F'_k)|=2^{k+1}-1$.
Letting $F_k=\AA_{F'_k,\prec,m}$, Lemma~\ref{lemma-ordtoall} implies $\wcol_{2k}(F_k)\ge 2^{k+1}-1$.
Moreover, by Lemma~\ref{lemma-reproa}, $F_k$ is a touching graph of comparable axis-aligned boxes in $\mathbb{R}^3$.

\item[(ii)] Let $H'_{k,t}$ be the graph obtained in Lemma~\ref{lemma-thin},
represented as the intersection graph of a $t$-thin $m$-shrinking sequence of intervals for $m=\binom{k+t}{t}$.
Let $\prec$ be the sizewise ordering of $H'_{k,t}$.
By Lemma~\ref{lemma-verass}, we have $\wcol_{\prec,k}(H'_{k,t})=|V(H'_{k,t})|=\binom{k+t}{t}$.
Letting $H_{k,t}=\AA_{H'_{k,t},\prec,m}$, Lemma~\ref{lemma-ordtoall} implies $\wcol_{2k}(H_{k,t})\ge \binom{k+t}{t}$.
Moreover, by Lemma~\ref{lemma-reproa}, $H_{k,t}$ is the intersection graph of a $t$-thin set of axis-aligned squares in $\mathbb{R}^2$.

\item[(iii)] Recall that by Lemma~\ref{lemma-thin}, the graph $H'_{k,2^d-1}$ is properly $2^d$-colorable.
Let $T$ be the scaffolding of $H_{k,2^d-1}$.
For each $v\in V(H'_{k,2^d-1})$, we can assign the color of $v$ to all vertices in $T(v)$, obtaining
a proper coloring of $H_{k,2^d-1}$ by $2^d$ colors.  Corollary~\ref{cor-tradecol} implies that $H_{k,2^d-1}$
can be represented as a touching graph of axis-aligned hypercubes in $\mathbb{R}^{d+2}$.
\end{itemize}
\end{proof}

\section{Conclusions}
In this paper we have provided upper bounds on the weak coloring number of $t$-thin intersection graphs of $d$-dimensional objects of different kinds.
Our bounds are qualitatively tight in several aspects.
We would like to mention a few open questions, beyond improving the proven upper and lower bounds:
\begin{itemize}
 \item What is the asymptotic behavior of the $k$-th weak coloring numbers of planar graphs?
It is known to be $O(k^3)$~\cite{vdHOdMQRS16} and $\Omega(k^2\log k)$~\cite{JM21}.
 \item What is the asymptotic behavior of the $k$-th strong coloring numbers of touching graphs of unit balls in $\mathbb{R}^d$?
It is known to be $O(k^{d-1})$ and $\Omega(k^{d/2})$.
\end{itemize}

\subsection*{Acknowledgments}
This research was carried out at the workshop on Generalized Coloring Numbers organized by Micha{\l} Pilipczuk and Piotr Micek in February 2021.
We would like to thank the organizers and all participants for creating a friendly and productive environment.
Special thanks go to Stefan Felsner for fruitful discussions.

\bibliographystyle{siam}
\bibliography{data.bib}

\begin{thebibliography}{10}

\bibitem{BonMur}
{\sc J.~Bondy and U.~Murty}, {\em Graph {T}heory with {A}pplications},
  North-Holland, New York, Amsterdam, Oxford, 1976.

\bibitem{dujmovic2019graph}
{\sc V.~Dujmovi{\'c}, P.~Morin, and D.~R. Wood}, {\em Graph product structure
  for non-minor-closed classes}, arXiv, 1907.05168 (2019).

\bibitem{subconvex}
{\sc Z.~Dvo{\v{r}}{\'a}k, R.~McCarty, and S.~Norin}, {\em Sublinear separators
  in intersection graphs of convex shapes}, arXiv, 2001.01552 (2020).

\bibitem{espsublin}
{\sc L.~Esperet and J.-F. Raymond}, {\em Polynomial expansion and sublinear
  separators}, European Journal of Combinatorics, 69 (2018), pp.~49--53.

\bibitem{covcol}
{\sc M.~Grohe, S.~Kreutzer, R.~Rabinovich, S.~Siebertz, and K.~Stavropoulos},
  {\em Coloring and covering nowhere dense graphs}, SIAM Journal on Discrete
  Mathematics, 32 (2018), pp.~2467--2481.

\bibitem{JM21}
{\sc G.~Joret and P.~Micek}, {\em Improved bounds for weak coloring numbers},
  (2021).

\bibitem{kierstead2003orderings}
{\sc H.~Kierstead and D.~Yang}, {\em Orderings on graphs and game coloring
  number}, Order, 20 (2003), pp.~255--264.

\bibitem{koebe}
{\sc P.~Koebe}, {\em Kontaktprobleme der {K}onformen {A}bbildung}, Math.-Phys.
  Kl., 88 (1936), pp.~141--164.

\bibitem{nesbook}
{\sc J.~Ne{\v{s}}et\v{r}il and P.~{Ossona de Mendez}}, {\em Sparsity (Graphs,
  Structures, and Algorithms)}, vol.~28 of Algorithms and Combinatorics,
  Springer, 2012.

\bibitem{PPSlecturenotes}
{\sc M.~Pilipczuk, M.~Pilipczuk, and S.~Siebertz}, {\em Lecture notes for the
  course “Sparsity” given at Faculty of Mathematics, Informatics, and
  Mechanics of the University of Warsaw, Winter semesters 2017/18 and 2019/20}.

\bibitem{vdHOdMQRS16}
{\sc J.~{van den Heuvel}, P.~{Ossona de Mendez}, D.~{Quiroz}, R.~{Rabinovich},
  and S.~{Siebertz}}, {\em {On the Generalised Colouring Numbers of Graphs that
  Exclude a Fixed Minor}}, arXiv e-prints,  (2016), p.~arXiv:1602.09052.

\bibitem{van2017generalised}
{\sc J.~van~den Heuvel, P.~{Ossona de Mendez}, D.~Quiroz, R.~Rabinovich, and
  S.~Siebertz}, {\em On the generalised colouring numbers of graphs that
  exclude a fixed minor}, European Journal of Combinatorics, 66 (2017),
  pp.~129--144.

\bibitem{zhu2009colouring}
{\sc X.~Zhu}, {\em Colouring graphs with bounded generalized colouring number},
  Discrete Math., 309 (2009), pp.~5562--5568.

\end{thebibliography}

\end{document}